\documentclass[11pt]{amsart}
\usepackage{epsfig,amsfonts,amsthm,amssymb,latexsym,amsmath}

\providecommand{\abs}[1]{\Big\lvert#1\Big\rvert}

\theoremstyle{plain}
\newtheorem{theorem}{Theorem}[section]
\newtheorem{corollary}[theorem]{Corollary}
\newtheorem{lemma}[theorem]{Lemma}
\newtheorem{proposition}[theorem]{Proposition}

\theoremstyle{remark}
\newtheorem{remark}[theorem]{Remark}

\theoremstyle{definition}
\newtheorem{definition}[theorem]{Definition}

\title[Oresme Polynomials and Their Derivatives]{Oresme Polynomials and Their Derivatives}
\vspace{5pt}

\author[G. Cerda-Morales]{\scriptsize Gamaliel Cerda-Morales$^{1}$ }
\date{}

\begin{document}
\maketitle

\vspace{-20pt}
\begin{center}
{\footnotesize $^1$Departamento de Matem\'atica, Universidad de Concepci\'on, Concepci\'on, Chile.\\
E-mail: gamaliel.cerda@usm.cl }\end{center}

\hrule

\begin{abstract}
We study the problem of generalization of Oresme numbers with a new sequence of numbers called Oresme polynomials. Moreover, by using the matrix methods for Oresme polynomials, we obtain the identities including the general bilinear index-reduction formula of these numbers. Further, Oresme polynomials that are natural extensions of the $k$-Oresme numbers are introduced and some relations for the derivatives of these polynomials in the form of convolution are proved.
\end{abstract}

\medskip
\noindent

\medskip
\noindent
\medskip

\hrule

\section{Introduction}\label{sec:0}
In \cite{3}, A.F. Horadam presented a history of number attributed to Nicole Oresme, namely the sequence $$\{O_{n}\}_{n\geq 1}=\left\lbrace \frac{n}{2^{n}}\right\rbrace=\left\lbrace \frac{1}{2},\frac{2}{4},\frac{3}{8},\cdots, \frac{n}{2^{n}},\cdots \right\rbrace.$$ 

The Oresme numbers have many interesting properties and applications in many fields of science (see, e.g., \cite{0,1,2,3}). The Oresme numbers $O_{n}$ are defined by the recurrence relation
\begin{equation}\label{e1}
O_{0}=0,\ O_{1}=\frac{1}{2},\ \ \ O_{n+2}=O_{n+1}-\frac{1}{4}O_{n},\ n\geq0.
\end{equation}

In \cite{2} the basic list of identities provided by A.F. Horadam \cite{3} is expanded and extended to several identities for some of the $k$-Oresme numbers. In this sense, the $k$-Oresme numbers, $\{A_{n}^{(k)}\}_{n\geq0}$, are defined by
\begin{equation}\label{e2}
A_{0}^{(k)}=0,\ A_{1}^{(k)}=\frac{1}{k},\ \ \ A_{n+2}^{(k)}=A_{n+1}^{(k)}-\frac{1}{k^{2}}A_{n}^{(k)},\ n\geq0.
\end{equation}

The following properties given for $k$-Oresme numbers play important roles in this paper (see \cite{2,3}). For example,
\begin{equation}\label{e4}
A_{n+2}^{(k)}=\left(\frac{k^{2}-1}{k^{2}}\right)A_{n+1}^{(k)}-\frac{1}{k^{4}}A_{n-1}^{(k)},\ n\geq1,
\end{equation}
\begin{equation}\label{e5}
\sum_{j=0}^{n}A_{j}^{(k)}=k^{2}\left(\frac{1}{k}-A_{n+2}^{(k)}\right),
\end{equation}
\begin{equation}\label{eo5}
\sum_{j=0}^{n}(-1)^{j}A_{j}^{(k)}=\frac{k^{2}}{2k^{2}+1}\left(-\frac{1}{k}+(-1)^{n+1}\left(A_{n+2}^{(k)}-2A_{n+1}^{(k)}\right)\right),
\end{equation}
\begin{equation}\label{ec5}
\sum_{j=0}^{n}A_{2j+1}^{(k)}=\frac{k^{2}}{2k^{2}+1}\left(\frac{k^{2}+1}{k}+\frac{k^{2}+1}{k^{2}}\left(A_{2n+1}^{(k)}-k^{2}A_{2(n+1)}^{(k)}\right)\right)
\end{equation}
and
\begin{equation}\label{e6}
A_{n+1}^{(k)}A_{n-1}^{(k)}-\left(A_{n}^{(k)}\right)^{2}=-\left(\frac{1}{k^{2}}\right)^{n}, n\geq 1.
\end{equation}

Using standard techniques for solving recurrence relations, the auxiliary equation, and its roots are given by 
$$x^{2}-x+\frac{1}{k^{2}}=0;\ \textrm{and}\ x=\frac{k\pm \sqrt{k^{2}-4}}{2k}.$$ Thus the Binet formula can be written as
\begin{equation}\label{binet1}
A_{n}^{(k)}=\frac{1}{\sqrt{k^{2}-4}}\left[\left(\frac{k+ \sqrt{k^{2}-4}}{2k}\right)^{n}-\left(\frac{k- \sqrt{k^{2}-4}}{2k}\right)^{n}\right]
\end{equation}
if $k^{2}-4>0$. Note that $A_{n}^{(2)}$ is the $n$-th Oresme number.

In this paper, the sequence of $k$-Oresme numbers $\{A_{n}^{(k)}\}$ is extended to the sequence of rational functions, that we will call Oresme polynomials $\{O_{n}(x)\}$ by replacing $k$ with a real variable $x$. It is also observed that numerous properties of Oresme polynomials admit straightforward proofs.

\section{Oresme Polynomials}\label{sec:1}

\begin{definition}
Let $x$ be a nonzero real variable. The sequence of Oresme polynomials is recursively defined as follows:
\begin{equation}\label{d1}
O_{n+1}(x)=\left\{ \begin{array}{lcc}
             \frac{1}{x} &   \mathrm{if}  & n=0,1,\\
             \\ O_{n}(x)-\frac{1}{x^{2}}O_{n-1}(x) &  \mathrm{if} & n \geq 2.
             \end{array}
   \right.
\end{equation}
\end{definition}
Some initial Oresme polynomials are
\begin{align*}
O_{0}(x)&=0,\\
O_{1}(x)&=\frac{1}{x},\\
O_{2}(x)&=\frac{1}{x},\\
O_{3}(x)&=\frac{1}{x^{3}}(x^{2}-1),\\
O_{4}(x)&=\frac{1}{x^{3}}(x^{2}-2),\\
O_{5}(x)&=\frac{1}{x^{5}}(x^{4}-3x^{2}+1),\\
O_{6}(x)&=\frac{1}{x^{5}}(x^{4}-4x^{2}+3).
\end{align*}
It is easy to see in Eq. (\ref{d1}) that, $O_{n}(3)=\frac{F_{2n}}{3^{n}}$ for each $n$, where $F_{n}$ is the $n$-th Fibonacci number. 

Solving the auxiliary Eq. (\ref{d1}), we get the roots $\lambda_{1}(x)=(x+\sqrt{x^{2}-4})/2x$ and its conjugate $\lambda_{2}(x)=(x-\sqrt{x^{2}-4})/2x$, where $x^{2}-4>0$. Note that $\lambda_{1}(3)=(3+\sqrt{5})/6$ and $\lambda_{2}(3)=(3-\sqrt{5})/6$ are the constant for the numbers $\frac{F_{2n}}{3^{n}}$.

In order to deduce the Binet formula for Oresme polynomials, we use the method of induction to get
\begin{equation}\label{b1}
O_{n}(x)=\frac{1}{\sqrt{x^{2}-4}}\left(\lambda_{1}^{n}(x)-\lambda_{2}^{n}(x)\right),
\end{equation}
where $\lambda_{1}(x)=(x+\sqrt{x^{2}-4})/2x$ and $\lambda_{2}(x)=(x-\sqrt{x^{2}-4})/2x$ with $x^{2}-4>0$. 

Returning to the sequence $O_{n}(3)$, we note that $x=3$ yielded a Fibonacci number pattern. Thus it is seen from (\ref{b1}) that 
\begin{equation}\label{exam1}
O_{n}(3)=\frac{1}{3^{n}\sqrt{5}}\left[\left(\frac{3+\sqrt{5}}{2}\right)^{n}-\left(\frac{3-\sqrt{5}}{2}\right)^{n}\right].
\end{equation}

The following result demonstrates the limit of the quotient of two consecutive terms of Oresme polynomials:

\begin{lemma}\label{l1}
For $x>2$,
\begin{equation}\label{lab1}
\lim_{n\rightarrow \infty}\frac{O_{n+1}(x)}{O_{n}(x)}=\lambda_{1}(x).
\end{equation}
\end{lemma}
\begin{proof}
By using the Binet formula (\ref{b1}), we find
\begin{align*}
\lim_{n\rightarrow \infty}\frac{O_{n+1}(x)}{O_{n}(x)}&=\lim_{n\rightarrow \infty}\frac{\lambda_{1}^{n+1}(x)-\lambda_{2}^{n+1}(x)}{\lambda_{1}^{n}(x)-\lambda_{2}^{n}(x)} \cdot \frac{\frac{1}{\lambda_{1}^{n}(x)}}{\frac{1}{\lambda_{1}^{n}(x)}}\\
&=\lim_{n\rightarrow \infty}\frac{\lambda_{1}(x)-\left(\frac{\lambda_{2}(x)}{\lambda_{1}(x)}\right)^{n}\lambda_{2}(x)}{1-\left(\frac{\lambda_{2}(x)}{\lambda_{1}(x)}\right)^{n}}.
\end{align*}
Since $\lambda_{2}(x)<\lambda_{1}(x)$ for every $x>2$, we get $$\lim_{n\rightarrow \infty} \left(\frac{\lambda_{2}(x)}{\lambda_{1}(x)}\right)^{n}=0$$ and the desired result is obtained.
\end{proof}

It can be observed that, for $x=2$ or $x=-2$, $$\lim_{n\rightarrow \infty} \abs{\frac{O_{n+1}(x)}{O_{n}(x)}}=1.$$ However, this limit does not exist for $x$ lying between $-2$ and $2$.

The matrix corresponding to Oresme polynomials is represented as a second-order matrix $M_{or}$ whose entries are the first three modified Oresme polynomials, i.e., $$M_{or}(x)=\left[\begin{array}{cc} 1&-\frac{1}{x^{2}}\\ 1& 0 \end{array}\right].$$

By induction, we can easily verify that the matrix $M_{or}(x)$ raised to the $n$-th power is given by the formula $$M_{or}^{n}(x)=\left[\begin{array}{cc} xO_{n+1}(x)&-\frac{1}{x}O_{n}(x)\\ xO_{n}(x)& -\frac{1}{x}O_{n-1}(x) \end{array}\right],$$ for any integer $n\geq 1$. 

Since $\det(M_{or}(x))=\frac{1}{x^{2}}$ and $\det(M_{or}^{n}(x))=\left(\det(M_{or}(x))\right)^{n}=\frac{1}{x^{2n}}$, we conclude that $$O_{n+1}(x)O_{n-1}(x)-O_{n}^{2}(x)=-\frac{1}{x^{2n}}.$$

Further, since $M_{or}^{n}(x)M_{or}^{m}(x)=M_{or}^{n+m}(x)$ for all $n,m\geq 1$, the left expression is \begin{align*}
&M_{or}^{n}(x)M_{or}^{m}(x)\\
&=\left[\begin{array}{cc} x^{2}O_{n+1}(x)O_{m+1}(x)-O_{n}(x)O_{m}(x)&\frac{1}{x^{2}}O_{n}(x)O_{m-1}(x)-O_{n+1}(x)O_{m}(x)\\ x^{2}O_{n}(x)O_{m+1}(x)-O_{n-1}(x)O_{m}(x)&\frac{1}{x^{2}}O_{n-1}(x)O_{m-1}(x)-O_{n}(x)O_{m}(x) \end{array}\right],
\end{align*}
while the expression on the right-hand side reduces to
$$M_{or}^{n+m}(x)=\left[\begin{array}{cc} xO_{n+m+1}(x)&-\frac{1}{x}O_{n+m}(x)\\ xO_{n+m}(x)& -\frac{1}{x}O_{n+m-1}(x) \end{array}\right].$$ Equating the corresponding entries for both matrices, we have $$O_{n+m}(x)=xO_{n}(x)O_{m+1}(x)-\frac{1}{x}O_{n-1}(x)O_{m}(x).$$

\begin{proposition}[General bilinear index-reduction formula]
For integers $a$, $b$, $c$, $d$ and $t$ with $a+b=c+d$, we have
\begin{equation}\label{gb1}
O_{a}(x)O_{b}(x)-O_{c}(x)O_{d}(x)=\frac{1}{x^{2t}}\left(O_{a-t}(x)O_{b-t}(x)-O_{c-t}(x)O_{d-t}(x)\right).
\end{equation}
\end{proposition}
\begin{proof}
By the Binet formula (\ref{b1}), for all integers $n$, $r$ and $s$, the following identity is easily verified: $$O_{n+r}(x)O_{n+s}(x)-O_{n}(x)O_{n+r+s}(x)=\frac{1}{x^{2n}}O_{r}(x)O_{s}(x).$$ Setting $n=c$, $r=a-c$ and $s=b-c$ in this identity, we find $$O_{a}(x)O_{b}(x)-O_{c}(x)O_{a+b-c}(x)=\frac{1}{x^{2c}}O_{a-c}(x)O_{b-c}(x).$$ Moreover, setting $n=c-t$, $r=a-c$ and $s=b-c$ in the same identity, we obtain $$O_{a-t}(x)O_{b-t}(x)-O_{c-t}(x)O_{a+b-c-t}(x)=\frac{1}{x^{2(c-t)}}O_{a-c}(x)O_{b-c}(x).$$ Comparing these two identities and using $d=a+b-c$, we complete the proof of the formula.
\end{proof}
Setting $a=n+1$, $b=m$, $c=n$, $d=m+1$ and $t=n-1$ in the general bilinear index-reduction formula,
we arrive at the following result
\begin{corollary}
For the integers $n$ and $m$ such that $m\geq n$, we obtain
\begin{equation}
O_{n+1}(x)O_{m}(x)-O_{n}(x)O_{m+1}(x)=\frac{1}{x^{2n+1}}O_{m-n}(x).
\end{equation}
\end{corollary}

At the end of this section, we deduce a product formula for Oresme polynomials.
\begin{proposition}\label{pp}
If $O_{n}(x)$ is the $n$-th Oresme polynomial, then
\begin{equation}
O_{n}(x)=\frac{1}{x^{n}}\prod_{1\leq k\leq n-1} \left(x-2\cos\left(\frac{k\pi}{n}\right)\right)
\end{equation}
\end{proposition}
\begin{proof}
Clearly, the degree of the Oresme polynomial $x^{n}O_{n}(x)$ is $n-1$ for $n\geq 1$. To deduce a formula for the zeros of $x^{n}O_{n}(x)$, we express $O_{n}(x)$ in terms of hyperbolic functions by using the Binet formula. Setting $x/2= \cosh z$, we find 
\begin{equation}\label{g0}
\begin{aligned}
O_{n}(x)&=\frac{1}{\sqrt{x^{2}-4}}\left(\lambda_{1}^{n}(x)-\lambda_{2}^{n}(x)\right)\\
&=\frac{1}{x^{n}}\left(\frac{\sigma_{1}^{n}(x)-\sigma_{2}^{n}(x)}{\sigma_{1}(x)-\sigma_{2}(x)}\right)\ \ \left(\sigma_{1,2}(x)=\frac{x\pm \sqrt{x^{2}-4}}{2}\right)\\
&=\frac{1}{x^{n}}\left(\frac{e^{nz}-e^{-nz}}{e^{z}-e^{-z}}\right)\\
&=\frac{1}{x^{n}}\cdot \frac{\sinh nz}{\sinh z}.
\end{aligned}
\end{equation}
Let $z=u+iv$, where the imaginary number $i=\sqrt{-1}$. Since $\sinh z \neq 0$, $O_{n}(x)$ is equal to zero only if $\sinh nz= 0$, i.e., $e^{2nz}=1$. We equate the real part to zero, i.e., $u=0$. Therefore, $\sinh inv= i \sin nv = 0$. Hence, $v=k\pi /n$ for any integer $k$ and, therefore, $z=ki\pi /n$. Consequently, for any integer $k$, we get $$\frac{x}{2}=\cosh \left(\frac{ki\pi}{n}\right)=\cos \left(\frac{k\pi}{n}\right).$$ As a result, the proposition follows for $1\leq k\leq n-1$.
\end{proof}

\begin{remark}
It is also observed that, for $O_{n}(x)$, the product formula mentioned in Proposition \ref{pp} gives an alternating expression for $O_{n}(x)$. Thus, e.g., $$O_{4}(x)=\frac{1}{x^{4}}\prod_{1\leq k\leq 3} \left(x-2\cos\left(\frac{k\pi}{4}\right)\right)=\frac{1}{x^{4}}(x^{3}-2x).$$
\end{remark}

\section{Derivative Sequences of Oresme Polynomials}\label{sec:2}
In this section, we study the sequences obtained by differentiating the Oresme polynomials. Several properties of these sequences and some relations between the Oresme polynomials and their derivatives are also presented.

Identity (\ref{b1}) can be rewritten as follows:
\begin{equation}\label{bn}
O_{n}(x)=\sum_{j=0}^{\lfloor \frac{n-1}{2}\rfloor}(-1)^{j}\binom{n-j-1}{j}x^{-2j}.
\end{equation}
As a result of differentiation of (\ref{bn}) with respect to $x$, we obtain
\begin{equation}\label{bn1}
O_{n}'(x)=\sum_{j=0}^{\lfloor \frac{n-2}{2}\rfloor}(-1)^{j+1}(2j)\binom{n-j-1}{j}x^{-2j-1},\ n\geq 2,
\end{equation}
with $O_{0}'(x)=0$ and $O_{1}'(x)=-1/x^{2}$. Further, several first derivatives of the Oresme polynomials with respect to $x$ have the form
\begin{align*}
O_{0}'(x)&=0,\\
O_{1}'(x)&=-\frac{1}{x^{2}},\\
O_{2}'(x)&=-\frac{1}{x^{2}},\\
O_{3}'(x)&=-\frac{1}{x^{4}}(x^{2}-3),\\
O_{4}'(x)&=-\frac{1}{x^{4}}(x^{2}-6),\\
O_{5}'(x)&=-\frac{1}{x^{6}}(x^{4}-9x^{2}+5),\\
O_{6}'(x)&=-\frac{1}{x^{6}}(x^{4}-12x^{2}+15).
\end{align*}
Different types of numerical sequences can be generated by substituting integers instead of the variable $x$ of balancing polynomials, such as
\begin{align*}
O_{n}'(3)&=\{0,-\frac{1}{9},-\frac{1}{9},-\frac{6}{81},-\frac{3}{81},\cdots\},\\
O_{n}'(4)&=\{0,-\frac{1}{16},-\frac{1}{16},-\frac{13}{256},-\frac{10}{256},\cdots\}.
\end{align*}

The following results give some interesting relationships between the Oresme polynomials and their first derivatives:

\begin{proposition}
Let $O_{n}'(x)$ denote the first derivative of $O_{n}(x)$. Then, for $n\geq 2$
\begin{equation}\label{g1}
O_{n}'(x)=\frac{x^{2}(2n-x^{2})O_{n}(x)-2nO_{n-2}(x)}{x^{3}(x^{2}-4)},\ x\neq 0,\pm2.
\end{equation}
\end{proposition}
\begin{proof}
Recall that $\lambda_{1}(x)=(x+\sqrt{x^{2}-4})/2x$ and $\lambda_{2}(x)=(x-\sqrt{x^{2}-4})/2x$ with $\lambda_{1}(x)\lambda_{2}(x)=1/x^{2}$. Their first derivatives are, respectively, $$\lambda_{1}'(x)=\frac{2}{x^{2}\Delta}\ \textrm{and}\ \lambda_{2}'(x)=-\frac{2}{x^{2}\Delta},$$ where $\Delta=\sqrt{x^{2}-4}$ and $\Delta' =x/\Delta$. Therefore, by using the Binet formula (\ref{g0}), we find
\begin{align*}
O_{n}'(x)&=\frac{d}{dx}\left(\frac{\lambda_{1}^{n}(x)-\lambda_{2}^{n}(x)}{\Delta}\right)\\
&=\frac{\left[(\lambda_{1}^{n}(x))'-(\lambda_{2}^{n}(x))'\right]\Delta-(\lambda_{1}^{n}(x)-\lambda_{2}^{n}(x))\Delta'}{\Delta^{2}}\\
&=\frac{\frac{2n}{x\Delta}\left[\lambda_{1}^{n-1}(x)+\lambda_{2}^{n-1}(x)\right](\lambda_{1}(x)-\lambda_{2}(x))-xO_{n}(x)}{\Delta^{2}}\\
&=\frac{\frac{2n}{x\Delta}\left[\lambda_{1}^{n}(x)-\lambda_{2}^{n}(x)\right]-\frac{2n}{x^{3}\Delta}\left[\lambda_{1}^{n-2}(x)-\lambda_{2}^{n-2}(x)\right]-xO_{n}(x)}{\Delta^{2}}\\
&=\frac{x^{2}(2n-x^{2})O_{n}(x)-2nO_{n-2}(x)}{x^{3}(x^{2}-4)},
\end{align*}
which completes the proof.
\end{proof}

In particular, for $x=3$, identity (\ref{g1}) reduces to $$O_{n}'(3)=\frac{1}{135}\left(9(2n-9)\frac{F_{2n}}{3^{n}}-2n\frac{F_{2(n-2)}}{3^{n-2}}\right).$$

\begin{proposition}[Convolved Oresme polynomials]
If $O_{0}'(x)=0$ and $O_{1}'(x)=-1/x^{2}$, we obtain
\begin{equation}\label{n2}
O_{n}'(x)+\frac{n}{x}O_{n}(x)=\sum_{j=1}^{n-1}O_{j}(x)O_{n-j}(x),\ n\geq 2.
\end{equation}
\end{proposition}
\begin{proof}
The proof is obtained by induction on $n$. Clearly, the result holds for $n=2$, since $$O_{2}'(x)+\frac{2}{x}O_{2}(x)=\frac{1}{x^{2}}=O_{1}(x)O_{1}(x).$$ As the inductive hypothesis, we assume that the formula is true for every $m\leq n$. To perform the inductive step, we differentiate the recurrence relation for Oresme polynomials with respect to $x$, $$O_{n+1}'(x)=O_{n}'(x)+\frac{2}{x^{3}}O_{n-1}(x)-\frac{1}{x^{2}}O_{n-1}'(x)$$ and apply the inductive hypothesis to get
\begin{align*}
O_{n+1}'(x)&+\frac{(n+1)}{x}O_{n+1}(x)\\
&=O_{n}'(x)+\frac{2}{x^{3}}O_{n-1}(x)-\frac{1}{x^{2}}O_{n-1}'(x)\\
&\ \ +\frac{(n+1)}{x}O_{n}(x)-\frac{(n+1)}{x^{3}}O_{n-1}(x)\\
&=\frac{1}{x}O_{n}(x)+\left[O_{n}'(x)+\frac{n}{x}O_{n}(x)\right]\\
&\ \ -\frac{1}{x^{2}}\left[O_{n-1}'(x)+\frac{(n-1)}{x}O_{n-1}(x)\right]\\
&=\frac{1}{x}O_{n}(x)+\sum_{j=1}^{n-1}O_{j}(x)O_{n-j}(x)-\frac{1}{x^{2}}\sum_{j=1}^{n-2}O_{j}(x)O_{n-1-j}(x)\\
&=\frac{1}{x}O_{n}(x)+\frac{1}{x}O_{n-1}(x)+\sum_{j=1}^{n-2}O_{j}(x)O_{n+1-j}(x)\\
&=\sum_{j=1}^{n}O_{j}(x)O_{n+1-j}(x).
\end{align*}
This completes the proof.
\end{proof}

\begin{remark}
It is observed that identities (\ref{g1}) and (\ref{n2}) together yield the following formula: $$\frac{x^{2}((n-1)x^{2}-2n)O_{n}(x)-2nO_{n-2}(x)}{x^{3}(x^{2}-4)}=\sum_{j=1}^{n-1}O_{j}(x)O_{n-j}(x),$$ for $n\geq 2$ and $x\neq 0,\pm2$.
\end{remark}

\begin{proposition}
For $n\geq 2$, we have
\begin{equation}\label{g2}
O_{n}'(x)+\frac{n}{x}O_{n}(x)=\sum_{j=0}^{\lfloor \frac{n-2}{2}\rfloor}\frac{(n-1-2j)}{x^{2j+1}}O_{n-1-2j}(x).
\end{equation}
\end{proposition}
\begin{proof}
We proof this by induction. It is clear that the claim is true for $n=2$. Let us suppose that the claim is true for $n$. If $n=2m$ is an even integer, by induction hypothesis we get $$O_{2m}'(x)+\frac{2m}{x}O_{2m}(x)=\sum_{j=0}^{m-1}\frac{(2m-1-2j)}{x^{2j+1}}O_{2m-1-2j}(x)$$ y $$O_{2m-1}'(x)+\frac{(2m-1)}{x}O_{2m-1}(x)=\sum_{j=0}^{m-1}\frac{(2m-2-2j)}{x^{2j+1}}O_{2m-2-2j}(x).$$ Now, we will show the equation is true for $n+1$ namely $2m+1$. By using (\ref{d1}) we get
\begin{align*}
O_{2m+1}'(x)+\frac{(2m+1)}{x}&O_{2m+1}(x)\\
&=O_{2m}'(x)+\frac{2}{x^{3}}O_{2m-1}(x)-\frac{1}{x^{2}}O_{2m-1}'(x)\\
&\ \ +\frac{(2m+1)}{x}O_{2m}(x)-\frac{(2m+1)}{x^{3}}O_{2m-1}(x)\\
&=\frac{1}{x}O_{2m}(x)+\left[O_{2m}'(x)+\frac{2m}{x}O_{2m}(x)\right]\\
&\ \ -\frac{1}{x^{2}}\left[O_{2m-1}'(x)+\frac{(2m-1)}{x}O_{2m-1}(x)\right]\\
&=\frac{1}{x}O_{2m}(x)+\sum_{j=0}^{m-1}\frac{(2m-1-2j)}{x^{2j+1}}O_{2m-1-2j}(x)\\
&\ \ -\frac{1}{x^{2}}\sum_{j=0}^{m-1}\frac{(2m-2-2j)}{x^{2j+1}}O_{2m-2-2j}(x)\\
&=\sum_{j=0}^{m}\frac{(2m+1-2j)}{x^{2j+1}}O_{2m+1-2j}(x).
\end{align*}
The similar proof can be given for the case when $n$ is an odd integer. Thus, the result follows for all $n\geq 2$.
\end{proof}

\begin{proposition}
For $n\geq 1$,
\begin{equation}\label{g3}
(n-1)O_{n}(x)-2nO_{n+1}(x)=xO_{n+1}'(x)-\frac{1}{x}O_{n-1}'(x).
\end{equation}
\end{proposition}
\begin{proof}
The proof is obtained by induction on $n$. Clearly, the result holds for $n=1$. As the inductive hypothesis, we assume that the formula is true for every $m\leq n$. To perform the inductive step, we differentiate the recurrence relation for Oresme polynomials with respect to $x$, $$O_{n+1}'(x)=O_{n}'(x)+\frac{2}{x^{3}}O_{n-1}(x)-\frac{1}{x^{2}}O_{n-1}'(x)$$ and apply the inductive hypothesis to get
\begin{align*}
xO_{n+2}'(x)&-\frac{1}{x}O_{n}'(x)\\
&=xO_{n+1}'(x)+\frac{2}{x^{2}}O_{n}(x)-\frac{1}{x}O_{n}'(x)\\
&\ \ -\frac{1}{x}O_{n-1}'(x)-\frac{2}{x^{4}}O_{n-2}(x)+\frac{1}{x^{3}}O_{n-2}'(x)\\
&=\left[xO_{n+1}'(x)-\frac{1}{x}O_{n-1}'(x)\right]+\frac{2}{x^{2}}\left[O_{n}(x)-\frac{1}{x^{2}}O_{n-2}(x)\right]\\
&\ \ -\frac{1}{x^{2}}\left[xO_{n}'(x)-\frac{1}{x}O_{n-2}'(x)\right]\\
&=\left[(n-1)O_{n}(x)-2nO_{n+1}(x)\right]+\frac{2}{x^{2}}\left[O_{n}(x)-\frac{1}{x^{2}}O_{n-2}(x)\right]\\
&\ \ -\frac{1}{x^{2}}\left[(n-2)O_{n-1}(x)-2(n-1)O_{n}(x)\right]\\
&=nO_{n+1}(x)-2(n+1)O_{n+2}(x).
\end{align*}
This completes the proof.
\end{proof}

\section{Conclusions}\label{sec:4}
The generalized $k$-Oresme sequences have been considered and studied many of their properties. Further, the Oresme polynomials that are the natural extension of $k$-Oresme numbers are defined and established many of their properties. As matrices can also be used to represent Oresme polynomials, a Oresme polynomial matrix $M_{or}$ of order 2 has formed.

\medskip

\end{document}